\documentclass[a4paper]{scrartcl}

\usepackage[english]{babel}
\usepackage[utf8x]{inputenc}
\usepackage{amsmath}
\usepackage{mathtools}
\usepackage{amsthm}

\usepackage[a4paper,top=3cm,bottom=2cm,left=3cm,right=3cm,marginparwidth=1.75cm]{geometry}

\usepackage{amsmath,amssymb}
\usepackage{graphicx}
\usepackage[colorinlistoftodos]{todonotes}
\usepackage[colorlinks=true, allcolors=blue]{hyperref}

\newcommand{\R}{\mathbb{R}}
\newcommand{\C}{\mathbb{C}}

\DeclareMathOperator{\Tr}{Tr}
\DeclareMathOperator{\diag}{diag}

\newenvironment{thm}[2][Theorem]{\begin{trivlist}
\item[\hskip \labelsep {\bfseries #1}\hskip \labelsep {\bfseries #2.}]}{\end{trivlist}}

\newenvironment{definition}[2][Definition]{\begin{trivlist}
\item[\hskip \labelsep {\bfseries #1}\hskip \labelsep {\bfseries #2.}]}{\end{trivlist}}

\title{Uniqueness of the Gaussian Orthogonal Ensemble}
\author{José Ángel Sánchez Gómez\footnote{Undergraduate student. Universidad de Guanajuato. jose.sanchez@cimat.mx}, Victor Amaya Carvajal\footnote{Undergraduate student. Universidad de Guanajuato. victor.amaya@cimat.mx}.}
\date{December, 2017. }
\begin{document}
\maketitle

\begin{abstract}
A known result in random matrix theory states the following: Given a random Wigner matrix $X$ which belongs to the Gaussian Orthogonal Ensemble (GOE), then such matrix $X$ has an invariant distribution under orthogonal conjugations.  The goal of this work is to prove the converse, that is, if $X$ is a symmetric random matrix such that it is invariant under orthogonal conjugations, then such matrix $X$ belongs to the GOE. We will prove this using some elementary properties of the characteristic function of random variables.
\end{abstract}

\section{Introduction}

We will prove one of the main characterization of the matrices that belong to the Gaussian Orthogonal Ensemble.  This work was done as a final project for the Random Matrices graduate course at the Centro de Investigación en Matemáticas, A.C. (CIMAT), M\'exico.

\section{Background I}
We will denote by $ \mathcal{M}_{n \times m} (\mathbb{F} ) $, the set of matrices of size $n \times m$ over the field $\mathbb{F}$. 
The matrix $I_{d \times d}$ will denote the identity matrix of the set $ \mathcal{M}_{d \times d} (\mathbb{F} )$. 

\begin{definition}{1}
Denote by $ \mathcal{O}(n)$ the set of orthogonal matrices of $\mathcal{M}_{n \times n} (\mathbb{R} )$. This means, if $O \in \mathcal{O}(n)$ then $O^{\intercal}O = O O^{\intercal}  = I_{n \times n}$. 
\end{definition}

\begin{definition}{2}
A symmetric matrix $X$ is a square matrix such that $X^{\intercal} = X$.  We will denote the set of symmetric $n \times n$ matrices by $\mathbb{S}_n$.
\end{definition}

\begin{definition}{3} \textbf{[Gaussian Orthogonal Ensemble (GOE)].} We say that a symmetric real matrix $B_n \in \mathcal{M}_{n \times n}(\R)$ is in the Gaussian Orthogonal Ensemble if, $B_n(j,k)$, with $1 \leq j < k \leq n$ are independent random random variables 
\[
B_n \coloneqq
\begin{pmatrix}
B_n(1,1) &\ldots& B_n(1,n) \\
\vdots &\ldots&  \vdots \\
B_n(n,1) &\ldots& B_n(n,n) 
\end{pmatrix}
\]
and 
\begin{equation*}
\begin{aligned}
B_(j,k) = B_n(k,j), \quad
B_n(j,k) \sim \mathcal{N}(0,1),\,  j \neq k,\quad
B_{n}(j,j) \sim \mathcal{N}(0,2).
\end{aligned}
\end{equation*}
\end{definition}

\begin{definition}{4}
Denote by $\text{GOE}(n)$ the set of $n \times n$ matrices that belong to the Gaussian Orthogonal Ensemble.
\end{definition}

\begin{definition}{5}
Let $X \in \mathcal{M}_{n \times n}(\R)$ be a random matrix. We say that $X$ has invariant distribution (or just invariant) under orthogonal conjugations if for every non-random orthogonal matrix $O \in \mathcal{O}(n)$, we have that $O X O^{\intercal} \overset{\mathcal{L}}{=}X.$
\end{definition}

\noindent \textbf{Observation 1.} It follows from the definition that, if we take $Z \in \mathcal{M}_{n \times n}(\R)$ whose entries are independent random variables $z_{i,j} \sim \mathcal{N}(0,1)$, then\[
X = \frac{1}{\sqrt{2}} (Z + Z^{\intercal}) \in \text{GOE}(n).
\]

\noindent \textbf{Observation 2.} Note that, given any $B_n \in \text{GOE}(n)$ there exists a matrix $Z = [z_{i,j}]_{i,j=1}^n$, with $z_{i,j} \sim \mathcal{N}(0,1)$, such that $B_n \overset{\mathcal{L}}{=} \frac{1}{\sqrt{2}}(Z+Z^{\intercal}).$

\vspace{0.1cm}
Now, lets state the theorem whose converge we would like to prove.
\begin{thm}{2.1} ([Invariance under conjurations]) 
For any given $O \in \mathcal{O}(n)$ (non-random) and $B_n \in \text{GOE}(n)$, then \[
O X O^{\intercal} \overset{\mathcal{L}}{=}X.
\]
\end{thm}
\begin{proof}
By the observation 2, we can write matrix $B_n \overset{\mathcal{L}}{=} \frac{1}{\sqrt{2}}(Z+Z^{\intercal})$, with $Z$ such that $Z_{i,j} \sim \mathcal{N}(0,1)$. Now, observe that,
\[
O B_n O^{\intercal} \overset{\mathcal{L}}{=}
\frac{1}{\sqrt{2}} (OZO^{\intercal}+OZ^{\intercal}O^{\intercal}) \overset{\mathcal{L}}{=} 
\frac{1}{\sqrt{2}} (OZ+Z^{\intercal}O^{\intercal}) \overset{\mathcal{L}}{=} 
\frac{1}{\sqrt{2}} (Z+Z^{\intercal}) \overset{\mathcal{L}}{=} B_n
\]
\end{proof}

\section{Background II}

We define the characteristic function of a random variable $R$ to be a function $\varphi_R:\R\rightarrow \C$ given by,
$$ \varphi_R(t)= \mathbb{E}\left[ e^{i tR}\right],\qquad \forall t\in \R.$$

In particular, given a random variable $R \sim \mathcal{N}(\mu, \sigma^2)$, its characteristic function is given by: $$ \varphi_{R}(t) =  \mathbb{E} [e^{itR}] = e^{i\mu t - \frac{1}{2} \sigma^2 t^2}.$$

\noindent  Another important thing to remember is that if two random variables $R_1, R_2$ are such that their characteristic function coincide in every point, then their distributions are the same. In other words, if 
$\varphi_{X}(t) =  \varphi_{Y}(t)$, for every $t \in \mathbb{R}$, then $X \sim Y$. 

\bigskip
Given a symmetric random matrix $X\in \mathbb{S}_d$ its characteristic function is defined by:
$$ C_X: \mathbb{S}_d\rightarrow \C$$
$$ C_{X} (M) = \mathbb{E} \left[ \exp\left\{i \Tr(X^{\intercal}M) \right\} \right], \qquad \forall M \in \mathbb{S}_d.$$

\noindent Where $\Tr(\cdot)$ is the trace operator.  Recall that the trace is invariant under cyclic permutations. From this, if $A\in \mathcal{M}_{d\times d}(\R)$ and for all $M \in \mathbb{S}_d$.
\begin{eqnarray*}
C_X(A^{\intercal}MA) &=& \mathbb{E}\left[\exp\left\{ i\Tr(X^{\intercal} A^{\intercal} M A) \right\} \right] \\
&=& \mathbb{E}\left[\exp\left\{i\Tr(A^{\intercal}X^{\intercal}A  M) \right\} \right] \\
&=&  \mathbb{E}\left[\exp\left\{i\Tr([A^{\intercal}X^{\intercal}A]^{\intercal}  M) \right\} \right]= C_{A^\intercal X A}(M).
\end{eqnarray*}
$$	$$
Notice that, if $A,B\in \mathbb{S}_d$, 
$$ \Tr(A^{\intercal} B)= \sum_{j=1}^d (A^{\intercal}B)_{jj}= \sum_{j,k=1}^d (A^\intercal)_{jk}B_{kj}= \sum_{j,k=1}^d A_{kj}B_{kj}. $$
Furthermore, if $A,B\in\mathbb{S}_d$,
$$ \Tr(A^{\intercal} B)= \sum_{j=1}^d A_{jj}B_{jj} + 2\sum_{1\leq j< k\leq d} A_{jk}B_{jk}.$$
From this, note that
$$ C_X(M)= \mathbb{E}\left[ \exp\left\{  i\Tr(X^{\intercal} M)  \right\} \right]=  \mathbb{E}\left[ \exp\left\{   i\sum_{j=1}^d M_{jj}X_{jj}+ 2 i\sum_{1\leq j< k\leq d} M_{kj}X_{kj}\right\} \right].$$
In particular, if the entries of the matrix $X$ are independent, the characteristic function of $X$ can be written as the product of characteristic functions of each entry, i.e.,
\begin{eqnarray*}
	C_X(M)&=&  \prod_{j=1}^d \mathbb{E}\left[ \exp\left\{ iM_{jj}X_{jj}\right\} \right] \cdot \prod_{1\leq j<k\leq d}\mathbb{E}\left[ \exp\left\{ 2iM_{kj}X_{kj}\right\} \right].
\end{eqnarray*}
It will be useful to keep this formula in mind when computing the value of the characteristic function on particular symmetric matrices. 

\section{The problem}

Our aim is to prove the following theorem using nothing more than the characteristic function of a random matrix.

\begin{thm}{4.1}
Let $X$ be a non-zero random symmetric $d\times d$ matrix such that it is invariant under orthogonal conjugations. Suppose that all the entries of $X$ are independent with finite variance. Then, there exist a matrix belonging to the Gaussian Orthogonal Ensemble $Y$, and real numbers $\mu\in\R$, $\sigma^2\geq0$ such that $X \overset{\mathcal{L}}{=} \mu I_{d \times d}+\sigma^2 Y $. 
\end{thm}

\vspace{0.2cm}
\noindent \textit{Proof}:  
We will divided our proof into three steps:
\begin{itemize}
\item[1.] Show that the elements inside the diagonal share the same distribution, and that all the elements outside the main diagonal share the same distribution. 
\item[2.] Show that it is enough to prove the result for square $2 \times 2$ matrices.
\item[3.] We will prove that the characteristic functions of the elements of the matrix $X$ correspond to normally distributed random variables.
\end{itemize}

\textbf{Step one:} Let us define the matrix $A_{(k,j)}^t$ as the matrix that has the value $t$ in the positions $(k,j)$ and $(j,k)$ and zeros everywhere else, for $1\leq k < j\leq d$ and $t\in \R$. If we calculate the the characteristic function of $A_{(k,j)}^t$ we obtain:
\begin{eqnarray*}
C_X (A_{(k,j)}(t)) &=&  \mathbb{E} \left[ \exp \left\{ i\Tr((A_{(k,j)}^t)^{\intercal}X) \right\} \right] \\
&=& \mathbb{E} \left[ \exp\left\{ 2itX_{kj} \right\} \right]\\
&=& \phi_{ X_{kj} } (2t).
\end{eqnarray*}
Now, let $k\neq j$ and $l\neq g$. There exists a permutation matrix $P\in \mathcal{O}(d)$ such that for all $t\in \R$, $P^{\intercal} A_{(k,j)}^t P= A_{(l,g)}^t$. From the invariance of $X$ with respect to orthogonal conjugations:

$$ \varphi_{X_{kj}}(2t)= C_X(A_{(k,j)}^t)= C_X(P^{\intercal} A_{(k,j)}^t P)=C_X(A_{(l,g)}^t)= \varphi_{X_{lg}}(2t),\qquad \forall t\in \R.$$ 
It follows that $X_{kj}\sim X_{lg}$. It means that all entries outside the main diagonal have the same distribution. In an analogous way, by conjugating we can prove that if $1\leq j< k \leq d$, $X_{kj}\sim -X_{kj}$. This implies the distribution of the entries outside of the main diagonal is symmetric. This is because all entries of the matrix have finite second moment, $\mathbb{E}(X_{kj})=0$. Now, let $A_i^t= \diag(t \cdot e_j)$, i.e. a matrix that has as only non-zero entry $t$ in the position $(j,j)$, for $1\leq j\leq d$ and $t\in \R$. From this,
\begin{eqnarray*}
C_X(A_j^t)&=& \mathbb{E}\left[ \exp\left\{i \Tr((A_j^t)^{\intercal}X) \right\}\right] \\
&=& \mathbb{E} \left[ \exp \left\{ itX_{jj} \right\} \right]\\
&=& \varphi_{ X_{jj} }(t),
\end{eqnarray*}
for all $1\leq j\leq d$ and $t\in \R$. Now, given $j\neq k$, there exists a permutation matrix $Q\in \mathcal{O}(d)$ such that $Q^{\intercal} A_j^t Q= A_k^t$.From this,
$$ \varphi_{X_{jj}}(t)= C_X(A_j^t)= C_X(Q^{\intercal}A_j^tQ )= C_X(A_k^t)= \varphi_{X_{kk}}(t) \qquad t\in \R.$$
Thus $X_{jj}\sim X_{kk}$ for all $j\neq k$. Then, all entries in the diagonal have the same distribution.
\\

\noindent \textbf{Step two:} In the last step, we have proved that all entries in the main diagonal have the same distribution, and the entries outside the diagonal share the same distribution. From this, it is possible to find the distribution of all entries of the matrix $X$ by finding the distribution of the entries of the sub-matrix,
$$ X_2= \begin{pmatrix}
	X_{11} & X_{12} \\
    X_{12} & X_{22} \\
\end{pmatrix}.$$

It can be proved that $X_2$ is invariant with respect to rotations in $\R^2$. For this, let $\theta\in \R$ and,

$$ Q'_\theta= \begin{pmatrix}
	M_\theta & 0 \\
    0  & I_{(d-2)\times (d-2)} \\
\end{pmatrix}\in \mathcal{M}_{d\times d}(\R),$$
where,
$$Q_{\theta}= \begin{pmatrix}
	\cos(\theta) & -\sin(\theta) \\
    \sin(\theta) & \cos(\theta).\\
\end{pmatrix} \in \mathcal{M}_{2\times 2}(\R).$$
The matrix $Q'_\theta$ is an orthogonal matrix that rotates the first two coordinates and keep fixed the others. Let $M\in\mathcal{S}_2$, where
$$M= \begin{pmatrix}
	a & b\\
    b & d\\
\end{pmatrix}\in \mathbb{S}_2. $$
We can embed $M$ into $\mathbb{S}_d$ by filling all entries outside the $2\times 2$ first sub-matrix with zeros:
$$ M'= \begin{pmatrix}
	M & 0 \\
    0 & 0 \\
\end{pmatrix} \in \mathcal{S}_d.$$

Now lets observe that,

$$ C_X(M')= \mathbb{E}\left[ \exp\left\{ i\Tr(M^{\intercal}X) \right\} \right]= \mathbb{E}\left[  \exp\left\{ i(aX_{11}+2bX_{12}+ dX_{22})  \right\} \right]= C_{X_2}(M).$$

By evaluating the characteristic function of $X$ on $M'$ and using the orthogonal in-variance, and considering that all entries outside the $2\times 2$ matrix in $M'$ are zero,

$$ C_{X_2}(M)= C_X(M')= C_X((Q'_\theta)^{\intercal} M' Q'_\theta) = C_{X_2}(Q^{\intercal}_\theta M Q_\theta) .$$


\noindent \textbf{Step three:} Since we have reduced the problem to the case of $2 \times 2$ matrices, we will keep $M$ and $Q_\theta$ to be the matrices defined in the previous step. We take $M_{\theta} = Q_{\theta} ^ {\intercal} M Q_{\theta}$, a $2 \times 2$ symmetric matrix. Name the entries of $M_{\theta}$ as follow
 $$  M_{\theta} =\begin{pmatrix}
   A & B \\
   B & D \\
  \end{pmatrix} ,
  $$ 
we can compute explicitly the values of $A,B,D$ in function the entries of $M$. We get the following:

\begin{eqnarray*}
A &=& \frac{a+d}{2} + \frac{a-d}{2} \cos(2 \theta) - b \sin(2 \theta), \\
B &=& \frac{a-d}{2} \sin(2 \theta) + b \cos(2 \theta), \\
D &=& \frac{a+d}{2} - \frac{a-d}{2} \cos(2 \theta) + b \sin(2 \theta).
\end{eqnarray*}

These expressions are valid for every $\theta \in \mathbb{R}$. If we calculate the derivatives of the last expressions with respect to $\theta$, what we get is the following:
\begin{equation} \label{eq:1}
\frac{dA}{d \theta}  = -2B,\qquad \frac{dB}{d \theta} = A-D,\qquad \frac{dD}{d \theta} = 2B
\end{equation}

Now, since $X_2$ is invariant under rotations, we have that:
$$ C_{X_2} (M_{\theta}) = C_{X_2}(M).$$

\noindent Computing the characteristic function on both matrices we obtain that,
\begin{equation}\label{eq:2}
\varphi_1(a)\varphi_2(2b)\varphi_1(d) = \varphi_1(A)\varphi_2(2B)\varphi_1(D),
\end{equation}

\noindent where $\varphi_1$ represents the characteristic function of the random variable $X_{11}$ (which is the same of $X_{22}$ since the have the same distribution). And $\varphi_2$ is the characteristic function of $X_{12}$.

Since entries of $X_2$ have finite variance $\varphi_1$ and $\varphi_2$ are at least two times derivable at zero and $\phi_j'(0)= i\mathbb{E}(X_{1j})$ for $j=1,2$. We can derive equation (\ref{eq:1}) with respect to $\theta$, we get:

\begin{equation} \label{eq:3}
0 = \varphi_1'(A) \varphi_2(2B) \varphi_1(D) \frac{dA}{d \theta} +
\varphi_1(A) \varphi_2'(2B) \varphi_1(D) 2 \frac{dB}{d \theta} +
\varphi_1(A) \varphi_2(2B) \varphi_1'(D) \frac{dD}{d \theta} 
\end{equation}

\noindent Note that since $\varphi_1(0) = \varphi_2(0)=1$, and the are continuous at zero, there exist an open set around zero where these functions are not zero. So, we can rewrite equation (\ref{eq:3}) in the following way: 

\begin{equation*}
\begin{split}
0 &= \frac{\varphi_1'(A)}{\varphi_1(A)}\varphi_1(A) \varphi_2(2B) \varphi_1(D) \frac{dA}{d \theta} + \frac{\varphi_2'(2B)}{\varphi_2(2B)}\varphi_1(A) \varphi_2(2B) \varphi_1(D) 2 \frac{dB}{d \theta} + \\
 &\quad+ \frac{\varphi_1'(D)}{\varphi_1(D)}\varphi_1(A) \varphi_2(2B) \varphi_1(D) \frac{dD}{d \theta}.
\end{split}
\end{equation*}

Now, replacing the values we found in (\ref{eq:1}) and dividing by $-B(A-D)$ (for those values $B \neq 0, A \neq D$),  we get that 

\begin{equation} \label{eq:4}
\frac{1}{B} \frac{\varphi_2'(2B)}{\varphi_2(2B)}  =  \frac{1}{A-D} \left[ \frac{\varphi_1'(A)}{\varphi_1(A)}  - \frac{\varphi_1'(D)}{\varphi_1(D)} \right]
\end{equation}

We observed that the left hand side of (\ref{eq:4}) only depends on $B$, whereas the right hand side depends in both $A$ and $D$. By this observation we can conclude that it has to be constant. Then, there exists $k \in \mathbb{R}$ such that:
$$  \frac{1}{B} \frac{\varphi_2'(2B)}{\varphi_2(2B)}  = -k,$$

Observe that, since $\varphi_2$ is the characteristic function of a symmetric random variable, the co-domain of $\varphi_2$ is contained in $\R$. This implies $k\in\R$. Now, lets take $x = 2B$. We have that $\varphi_2'(x) = -\frac{xk}{2} \varphi(x)$.  This is an easy real-valued ODE whose solution is given by $\varphi_2(x) = B_0 e^{- \frac{kx^2}{2}}$, for some $B_0 \in \mathbb{R}$. Now, as $\varphi_2$ is a characteristic function it is true that $\varphi_2(0)=1$. This give us as a result that $B_0=1$. So, we have that \[
\varphi_2(x) =  \exp(
-\frac{k}{2}\frac{x^2}{2}). \]
At first, we have this is true at a neighborhood of $0\in \R$, but we can extend this solution on $\R$. Observe that if $k< 0$ then $|\varphi_2(t)|\rightarrow \infty$ when $|t|\rightarrow \infty$. This leads to a contradiction since characteristic functions have bounded image. Then, if $k >0$ and $\varphi_2$ is the characteristic function of a normal random variable with parameters $\mathcal{N}(0, k/2)$. Given $k=0$, $\varphi_2$ is the characteristic function of a singular distribution at $0$. 

On the other hand, taking $D=0$ in equation (\ref{eq:4}), it follows that
$$ \frac{1}{A} \left[ \frac{\varphi_1'(A)}{\varphi_1(A)}  - \frac{\varphi_1'(0)}{\varphi_1(0)} \right].$$

\noindent We know that $\varphi_1(0)=1$ and that $\varphi'(0)= i \mu$, where $\mu= \mathbb{E}(X_{11})$. So,

$$ \frac{\varphi_1'(A)}{\varphi(A)} = -Ak + i \mu \qquad \Rightarrow \qquad \varphi_1'(A) = (-Ak + i \mu) \varphi_1(A).$$

The solution to this ODE is $\varphi_1(x) = e^{i \mu x - \frac{kx^2}{2}}$. which is nothing more than the characteristic function of a random variable with distribution $\mathcal{N}(\mu,k).$

Since $k$ is non-negative we can write it as $k = 2 \sigma^2$. So we can write that characteristic functions in the following way:
$$ \varphi_1(x) =	e^{i \mu x -  \sigma^2 x^2},\qquad \varphi_2(x) = e^{- \sigma^2 \frac{x^2}{2}}.
$$

\noindent which means that $X_{11} \sim \mathcal{N}(\mu,2 \sigma^2)$ and $X_{12} \sim \mathcal{N}(0, \sigma^2)$. \\

In summary, we have proved that $X_{jj}\sim \mathcal{N}(\mu, 2\sigma^2)$ and $X_{kj}\sim \mathcal{N}(0,\sigma^2)$. Now, if the take a GOE matrix $Y$, i.e., $Y_{ii} \sim \mathcal{N}(0,2)$ and $Y_{jk} \sim \mathcal{N}(0, 1)$, for every $j \neq k$, we conclude that:

$$ X \overset{\mathcal{L}}{=} \sigma^2 Y + \mu I_d.$$

We see that X is clearly invariant under orthogonal conjugations.  \hfill $\Box$

\section{Conclusions}
There are important remarks that are worth stressing out about this proof. First, we find interesting that it unfolded by using  elementary knowledge of basic properties of characteristic functions, differential equations and linear algebra. It is a simple proof for this result.

Second, it states a generalization of the standard result. This result is proved usually for matrices that have zero-mean entries. This proof allow us to study general matrices that are invariant with respect to orthogonal conjugations.

The independence of the entries is a key requirement for this proof. One example of a distribution that is invariant under orthogonal conjugation is the Haar distribution on $\mathcal{O}(d)$. This distribution does not hold the entry-wise independence since orthogonal matrices satisfy relations between them. In terms of this work, it is  important to write the characteristic function of the matrix in terms of the product of the characteristic function of each entry.

One possible extension of this project is to find the distribution of Wishart matrices by applying the same characteristic-function approach.

\end{document}